\documentclass[11pt,a4paper,reqno]{amsart}
\usepackage{amsthm,amsmath,amssymb}
\usepackage[english]{babel}
\usepackage{graphicx}
\usepackage{enumitem}
\usepackage{tikz}
\usetikzlibrary{matrix}
\usepackage{hyperref}\hypersetup{colorlinks, citecolor=black,filecolor=black,linkcolor=black,urlcolor=black}

\newtheorem*{thm*}{Theorem}
\newtheorem*{mthm*}{Main Theorem}
\newtheorem{lem}{Lemma}
\newtheorem*{dfn*}{Definition}
\newtheorem*{question*}{Question}
\newtheorem{dfn}{Definition}
\newtheorem{ex}{Example}
\newtheorem*{ex*}{Example}
\newtheorem{rmk}{Remark}

\newcommand{\ctg}{\operatorname{ctg}}

\makeatletter
\def\author@andify{%
  \nxandlist {\unskip ,\penalty-1 \space\ignorespaces}%
    {\unskip {} \@@and~}%
    {\unskip \penalty-2 \space \@@and~}%
}
\makeatother

\title[Metrics that are Douglas \& generalized Berwald in dim. two]{Finsler metrics that are both Douglas and generalized Berwald in dimension two}
\author{Nina Bartelme\ss}
\author{Julius Lang}

\address{Affiliation and addresses:}
\address{Nina Bartelme\ss, Julius Lang, Faculty of Mathematics and Informatics, Friedrich-Schiller University Jena, Ernst-Abbe-Platz 2, 07743 Jena, Germany}
\email{nina.bartelmess@uni-jena.de, julius.lang@uni-jena.de}

\begin{document}

	\begin{abstract}
	We proof that in dimension two, a Finsler metric is Douglas and generalized Berwald, if and only if it is Berwald or a Randers metric $\alpha + \beta$, where $\beta$ is closed and is of constant length with respect to $\alpha$.
	\end{abstract}
	
	\maketitle

\section{Introduction}
	Berwald metrics are maybe the most important class of Finsler metrics. They contain Riemannian metrics as a subclass.
	\begin{dfn}	
	A Finsler metric $F$ is called \textit{Berwald}, if one of the following two equivalent properties holds:
	\begin{enumerate}[label=(\alph*)]
	\item Its geodesics coincide with the geodesics of an affine connection.
	\item There exists an affine, torsion free connection, whose parallel transport preserves the metric $F$.
	\end{enumerate}
	\end{dfn}
	
	There are two immediate generalizations of this definition:
	\begin{dfn}
	A Finsler metric $F$ is called
	\begin{enumerate}[label=(\alph*)]
	\item Douglas, if its geodesic coincide with the geodesics of an affine connection ${}^D\nabla$ up to orientation preserving reparametrization.
	\item generalized Berwald, if there exists an affine connection ${}^{gB} \nabla$ (possibly with torsion), whose parallel transport preserves the metric $F$.
	\end{enumerate}
	\end{dfn}
	
	Both generalizations are classical \cite{DefinitionDouglas, DefinitionGB} and have been studied in their own interest (e.g. \cite{DouglasReference1,DouglasReference2,GBReference2,GBReference1}). However the following natural question has never been studied: 
	\begin{question*}
	What can be said of Finsler metrics, that are both Douglas and generalized Berwald?
	\end{question*}	

\begin{figure}[ht]\centering 
		\begin{tikzpicture}
			\draw (0,1.2) circle (0.5) node{\Large\textbf{B}};
			\draw (-1,1) ellipse (3 and 2) node[yshift=12,xshift=-50]{\Large\textbf{D}};
			\draw (1,1) ellipse (3 and 2) node[yshift=12,xshift=50]{\Large\textbf{gB}};
			\node[] at (0,2.3) {\Huge ?};
			\draw (0,0) ellipse (0.85 and 0.4) node{\scriptsize Randers ex.};
		\end{tikzpicture}\caption{Two generalizations of Berwald metrics}
		\label{FigureIntersections}
		\end{figure}
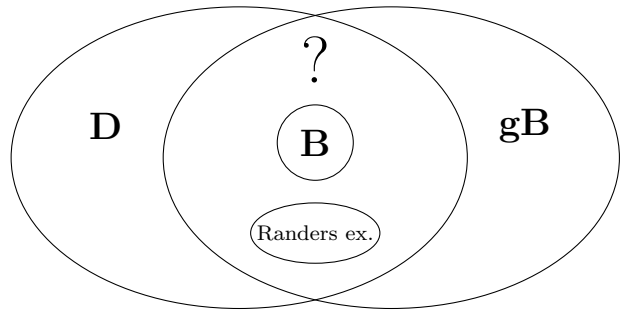	
	
		Clearly, every Berwald metric is Douglas and generalized Berwald.
		There is one more trivial example in the intersection of the two classes:
	
	\begin{ex}
	A Randers metric $F=\alpha + \beta$ is
	\begin{enumerate}[label=(\alph*)]
	\item \label{RandersDouglasCondition} Douglas, if and only if $\beta$ is closed.
	\item \label{RandersGBCondition} generalized Berwald, if and only if the length of $\beta$ with respect to $\alpha$ (in local coordinates $\alpha^{ij}\beta_i \beta_j$) is constant.
	\end{enumerate}
	Fact \ref{RandersDouglasCondition} is almost trivial and first appeared in \cite{BascsoMatsumotoAGeneralizationOfTheNotion}, see also \cite{MatveevOnProjectiveEquivalence}.
	Fact \ref{RandersGBCondition} was proven in \cite[Theorem 2]{CsabaRanders}, but also follows from the following three facts:
	
	\begin{itemize}
	\item Monochromatic characterization of generalized Berwald metrics \cite{BartelmessMatveev}: A Finsler metric is generalized Berwald, if and only if each two tangent spaces $(T_xM, F(x, \cdot))$ are linearly isometric.
	\item Zermelo navigation \cite{ZermeloNavigation}: For any Randers metric $F=\alpha + \beta$, there is a Riemannian metric $h$ and a vector field $W$ on $M$, such that the unit balls of $F$ are the unit balls of $h$, shifted by the vector field $W$. Furthermore, 
	$h(W,W)=h_{ij}W^i W^j=\alpha^{ij}\beta_i \beta_j=|\beta|_\alpha^2$, where $\alpha^{ij}$ is the matrix inverse of $\alpha_{ij}$.
	\item Let $(Q_{ij}), (\tilde Q_{ij})$ be positive definite $(n\times n)$-matrices and $v, \tilde v \in \mathbb R^n$ be vectors. The shifted ellipsoids
	$\{y + v  \mid Q_{ij}y^iy^j=1\}\subseteq \mathbb R^n$ and $\{y + \tilde v  \mid \tilde Q_{ij}y^iy^j=1\} \subseteq \mathbb R^n$
	can be mapped to each other by a linear transformation, if and only if $Q_{ij}v^iv^j = \tilde Q_{ij} \tilde v^i \tilde v^j$.
	\end{itemize}
	\end{ex}

	We expect that, generally, the class of metrics, that are both Douglas and generalized Berwald, contains many interesting examples. In this paper however, we show that in dimension two, this is not the case:
\begin{thm*} In dimension two, any fiber-globally defined Finsler metric, that is both Douglas and generalized Berwald, is Berwald or a Randers metric.
\end{thm*}

\begin{rmk} The proof essentially uses that the Finsler metric is strictly convex and defined fiber-globally, that is on the whole tangent spaces $T_x M$ for $x \in M$. 
\end{rmk}

\textbf{Structure.} We first derive in Section \ref{SectionNecessaryPDE} for all dimensions a linear PDE system (\ref{PDEsystemExtended}), that a Finsler metric must satisfy if it is Douglas and generalized Berwald with respect to connections ${}^D\nabla$ and ${}^{gB}\nabla$.
The system involves only $y$-derivatives and thus lives on a fixed $T_xM$. The coefficients are given in terms of the Christoffel symbols of the two connections.

In Section \ref{SectionDimension2} we consider the system in dimension two, where it is only one equation. Coincidentally, this equation coincides with equation (2.10) from \cite{Samaneh}, where it was obtained as a necessary condition for a Douglas metric to admit a conformally equivalent Douglas metric.
Following the lines from \cite{Samaneh}, we show that, if the system admits a fiber-global, non-Berwald Finsler metric as a solution, the coefficients of the equation must be of a special form. In this case, all solutions can be found explicitly and are Randers metrics. The theorem follows.

{~}

\textbf{Acknowledgement.} The question of studying the intersection of Douglas and generalized Berwald metrics appeared during a DAAD-funded visit at the University of Debrecen. We thank Vladimir Matveev, Csaba Vincze, Tianyu Ma, Samaneh Saberali and Jan Schumm for discussions.
\pagebreak


\section{A necessary linear PDE system for all dimensions}\label{SectionNecessaryPDE}
\begin{dfn}{~}
\begin{enumerate}[label=(\alph*)]
\item A \textit{strictly convex Finsler metric} is a function $F: TM\backslash 0 \to \mathbb R_{> 0}$, such that
	\begin{itemize}
	\item $F(x,\lambda y) = \lambda F(x,y)$ for all $\lambda > 0$.
	\item $F|_{TM \backslash 0}$ is smooth and the matrix $g_{ij}|_{(x,y)} :=\tfrac12 \frac{\partial^2 F^2}{\partial{y^i} \partial{y^j}} \Big|_{(x,y)}$ is positive definite for all $(x,y) \in TM \backslash 0$.
	\end{itemize}
	
	\item The geodesic spray of $F$ is the vector field $S_F = y^i \partial_{x^i}-2G^i(x,y) \partial_{y^i}$ on $T M \backslash 0$ with
	$$2G^i(x, y) = g^{ij}\Big( E_{x^\ell y^j}y^\ell - E_{x^j} \Big),$$ where $E=\tfrac12 F^2$ and $(g^{ij})$ is the matrix inverse of $(g_{ij})$.

\end{enumerate}
\end{dfn}

	A Finsler metric $F$ is \textit{Douglas} with respect to an affine, torsion-free connection ${}^D \nabla$, if and only if each geodesic of $F$ is after an orientation preserving reparametrization a geodesic of ${}^D\nabla$. Let ${}^D \Gamma^i_{jk}$ be the Christoffel symbols of ${}^D \nabla$ in local coordinates. Then, this is the case, if and only if in all local coordinates there exists a function $\rho: TM \backslash   0\to \mathbb R$, such that
	\begin{equation}\label{SprayEquationDouglas}
	2G^i(x,y) = {}^D\Gamma^i_{jk}(x)y^j y^k + \rho(x,y) y^i.	
	\end{equation}

	Recall that a vector field $X: [a,b] \to TM$ along a curve $c: [a,b] \to M$ is said to be \textit{parallel} with respect to an affine connection $\nabla$, if $\nabla_{\dot c} Y=0$. In local coordinates, if $Y=Y^i \partial_{x^i}$, this condition is given by
	$$\dot X^i(t) + \Gamma_{jk}^i\Big(c(t) \Big) \dot c^j(t) X^k(t)=0 \qquad \forall t \in [a,b],  1 \leq i \leq n.$$
	The rate of change of the Finsler metric along such a parallel vector field is, suppressing the obvious arguments:
	\begin{align}\label{EquationRateOfChange}
	\frac{d}{dt} \Big( F(c(t),X(t))\Big) & = F_{x^j} \dot c^j + F_{y^i} \dot X^i\nonumber\\
	& = F_{x^j}\dot c^j(t) - F_{y^i} \Gamma^i_{jk} \dot c^j X^k\\
	& = \Big[F_{x^j} - F_{y^i} \Gamma^i_{jk}  X^k \Big] \dot c^j.\nonumber
	\end{align}	 

	A Finsler metric is \textit{generalized Berwald} with respect to a connection ${}^{gB}\nabla$ with Christoffel symbols ${}^{gB}\Gamma^i_{jk}$ in local coordinates, if and only if the rate of change (\ref{EquationRateOfChange}) vanishes for all possible curves and all parallel vector fields along these curves. As for any prescribed data $(c(0),\dot c(0), X(0))$, there exist such a curve an a parallel vector field, we conclude that $F$ is generalized Berwald with respect to ${}^{gB}\nabla$, if and only if 
	$$F_{x^j} = F_{y^i} {}^{gB}\Gamma^i_{jk} y^k \qquad \qquad (1 \leq j \leq n)$$
	or equivalently for $E=\tfrac12 F^2$
	\begin{equation}\label{GBCondition}
	E_{x^j} = E_{y^i} {}^{gB}\Gamma^i_{jk} y^k \qquad \qquad (1 \leq j \leq n).
	\end{equation}
	
	Plugging (\ref{GBCondition}) in the definition of the spray coefficients, we can eliminate all $x$-derivatives and obtain:
	\begin{align}\label{SprayCoefficientsGB}
	2G^i  & = g^{ij}  \Big[
	\Big( E_{y^m} {}^{gB}\Gamma^m_{\ell k} y^k \Big)_{y^j}y^\ell
	- E_{y^m} {}^{gB}\Gamma^m_{jk} y^k
	\Big] \nonumber\\
	& = g^{ij}  \Big[
	 g_{jm} {}^{gB}\Gamma^m_{\ell k} y^k y^\ell
	 + E_{y^m}{}^{gB}\Gamma^m_{\ell j} y^\ell
	- E_{y^m} {}^{gB}\Gamma^m_{jk} y^k \Big]\\
	& = {}^{gB}\Gamma^i_{k \ell} y^k y^\ell + g^{ij} E_{y^m}\Big[{}^{gB}\Gamma^m_{\ell j} y^\ell - {}^{gB}\Gamma^m_{j \ell}  \Big]y^\ell. \nonumber
	\end{align}
	
	Combining the above two formulas for the spray coefficients and
	introducing the following notations for the difference of the symmetrized Christoffel symbols of the two connections and the torsion
	$$\Gamma^i_{jk} := \tfrac12 ({}^{gB}\Gamma^i_{jk}+{}^{gB}\Gamma^i_{kj}) - {}^{D}\Gamma^i_{jk}
	\qquad
	T^i_{jk}:={}^{gB}\Gamma^i_{jk}-{}^{gB}\Gamma^i_{kj}	
	$$
	$$\Gamma^i_{j}:=\Gamma^i_{jk}y^k \qquad \Gamma^i := \Gamma^i_{jk}y^jy^k
	\qquad T^i_{j}:=T^i_{jk}y^k,$$
	we obtain:

\begin{lem}\label{LemmaAllDimensionPDEs}
A Finsler metric $F$, that is Douglas with respect to an affine, torsion-free connection ${}^D\nabla$ and generalized Berwald with respect to an affine connection ${}^{gB}\nabla$, satisfies the linear PDE system
	 \begin{equation}\label{PDEsystem}
	  F_{y^k y^i} \Gamma^i -  F_{y^i} T^i_k=0 \qquad\qquad (1 \leq k \leq n).
	 \end{equation}
	The system (\ref{PDEsystem}) is equivalent to the system
	\begin{equation}\label{PDEsystemExtended}
	  F_{y^k y^i} (2\Gamma^i_s+ T^i_s) - F_{y^s y^i} (2 \Gamma^i_k + T^i_k) - 2F_{y^i}T^i_{ks} =0 \qquad (1\leq k < s \leq n).
	 \end{equation}
\end{lem}
\begin{proof}
Let $F$ be Douglas with respect to a torsion-free affine connection ${}^{D}\nabla$ and Douglas with respect to an affine connection ${}^{gB}\nabla$. Then by subtracting the two formulas (\ref{SprayEquationDouglas}) and (\ref{SprayCoefficientsGB}) for the spray coefficients of $F$ from each other, we get that for some function $\rho: TM \backslash 0 \to M$
\begin{align*}
0 &= \Gamma^i + g^{ij} E_{y^m}T^m_{\ell j} y^\ell - \rho y^i\\
& = \Gamma^i - g^{ij} E_{y^m}T^m_{j}- \rho y^i.
\end{align*}
We have used that the torsion $T^i_{jk}$ is antisymmetric in the lower indices. Contracting the equation by $g_{ki}$, we obtain
\begin{align*}
0 &=  g_{ki}\Gamma^i -  E_{y^m}T^m_{k} - \rho g_{ki}y^i\\
& = (FF_{y^k y^i}+F_{y^k} F_{y^i}) \Gamma^i - F F_{y^m} T^m_k  - \rho F F_{y^k}\\
& =  F(F_{y^k y^i} \Gamma^i -  F_{y^m} T^m_k)  +(F_{y^i} \Gamma^i- \rho F ) F_{y^k} 
\end{align*}
where we have used that by homogeneity $g_{ki}y^i = E_{y^k}$, $E_{y^k}=FF_{y^k}$ and $g_{ki}=FF_{y^k y^i} + F_{y^k}F_{y^i}$. Contracting with $y^k$ and using the homogeneity property $F_{y^k}y^k=F$ and $F_{y^ky^i}y^k=0$, we see that the second summand of the last equation must vanish, that is
$$F_{y^i} \Gamma^i- \rho F  = 0.$$
As $F$ is not vanishing on $TM \backslash 0$, the remaining part gives us the asserted equations
\begin{equation}
F_{y^k y^i} \Gamma^i -  F_{y^i} T^i_k=0. \tag{\ref{PDEsystem}}
\end{equation}
To see that this system implies (\ref{PDEsystemExtended}), differentiate
by $y^s$
	$$F_{y^k y^s y^i} \Gamma^i + F_{y^k y^i} 2\Gamma^i_s - F_{y^s y^i}T^i_k - F_{y^i}T^i_{ks}=0$$
 and taking the part antisymmetric in $(k,s)$, we obtain the second asserted system 
\begin{equation}
F_{y^k y^i} (2\Gamma^i_s+ T^i_s) - F_{y^s y^i} (2 \Gamma^i_k + T^i_k) - 2F_{y^i}T^i_{ks} =0.
\tag{\ref{PDEsystemExtended}} 
\end{equation}
Conversely, contracting the system (\ref{PDEsystemExtended}) with $y^s$, we obtain back (the double of) the system (\ref{PDEsystem}). Thus the two systems are indeed equivalent.
\end{proof}

\begin{rmk}Note that the PDE system (\ref{PDEsystemExtended}) contains only $y$-derivatives and thus can be seen as a system on a fixed tangent space $T_xM$. It formally consists out of $\tfrac{n(n-1)}{2}$ equations on $F$, supplemented by $n$ homogeneity equations $F_{ij}y^j =0$, so that the total number of equations equals the number $\tfrac{(n+1)n}{2}$ of 2nd order derivatives. Thus, if the coefficient matrix of the second order derivatives is non-degenerate (which is generically the case), the system can be solved for the highest order derivatives and be written into Cauchy-Frobenius form. Then for fixed connection data $(\Gamma, T)$ and a initial data $F(x_0,y_0)$ for some $(x_0,y_0) \in TM$, there is a unique solution $T_{x_0}M$. 
\end{rmk}

\begin{rmk}\label{RemarkSymmetrization}
If a Finsler metric $F$ solves the system (\ref{PDEsystemExtended}), so does its symmetrization $F_s(x,y):=\tfrac12 (F(x,y)+F(x,-y))$.
\end{rmk}

\section{Dimension two and proof of the theorem}\label{SectionDimension2}
	Before proving the theorem on the 2-dimensional case by investigating the system (\ref{PDEsystem}), let us collect some straight forward formulas.  Because of the homogeneity, it will be useful to work in polar coordinates for the fibers with respect to the local coordinates, that is $y^1 = r \cos \theta$ and $y^2 = r \sin \theta$. In this coordinates, any Finsler metric is of the form
	$$F(x_1,x_2,r,\theta) = r f(x_1,x_2, \theta)$$
	 for some smooth function $f$ that is $2\pi$-periodic in its third argument. By the standard relations
	$$\partial_{y^1} = -\tfrac{1}{r} \sin \theta \partial_\theta + \cos \theta\partial_r \qquad \text{and}\qquad \partial_{y^2} = \tfrac{1}{r}\cos \theta \partial_\theta + \sin \theta \partial_r$$
	one calculates that the second order $y$-derivatives of $F$ are given by
	 \begin{equation}\label{EquationRelationTraceAndHessian}
	(F_{y^i y^j}) = \frac{f+f_{\theta \theta}}{r}\begin{pmatrix}
	\sin^2 \theta & - \sin \theta \cos \theta \\ - \sin \theta \cos \theta& \cos^2 \theta
	\end{pmatrix}.
\end{equation}		 
	 
	The next Lemma answers the question, under what condition on $f$ a function of the form $F(x_1,x_2,r, \theta)= r f(x_1,x_2,\theta)$ is actually a strictly convex Finsler metric.
	\begin{lem}\label{LemmaTwoDimensionalWhenStrictlyConvex}
	The function $F(x_1,x_2,r,\theta)= r f(x_1,x_2,\theta)$ is a strictly convex Finsler metric, if and only if $f>0$ and $f+f_{\theta \theta}>0$. 
	\end{lem}	
	\begin{proof}
	By  chain rule it is $g_{ij}=F_{y^i} F_{y^j} + F F_{y^i y^j}$ and as a consequence the determinant of the fundamental tensor is given by $ \det(g_{ij}) = f^3(f+f_{\theta \theta })$.
	 Indeed, we have
	\begin{align*}		
	\det(g_{ij})= &  \det(F_{y^i} F_{y^j} + F F_{y^i y^j})\\
	 =&\underbrace{\det(F_{y^i} F_{y^j})}_{=0}
	 +\underbrace{F^2\det(F_{y^i y^j})}_{=0}\\
	 &+ F(F_{y^1}^2 F_{y^2y^2}+ F_{y^2}^2 F_{y^1y^1} -2 F_{y^1}F_{y^2} F_{y^1y^2})\\
	=&   f^3(f+f_{\theta \theta}).
	\end{align*}
	The matrix $g_{ij}$ is positive definite, if and only if $g_{11}= F_{y^1}^2 + FF_{y^1 y^1}$ and $\det(g_{ij})$ are positive.
	Thus assuming $f$ is positive and using (\ref{EquationRelationTraceAndHessian}), $F$ is strictly convex if and only if $f+f_{\theta\theta}$ is positive.
	\end{proof}	 
	 
	 Let us now prove the theorem:\vspace{-2mm}
\begin{proof}[Proof of the theorem]
Suppose $F$ is Douglas and generalized Berwald with respect to ${}^D\nabla$ and ${}^{gB}\nabla$. If $F$ is  not Berwald, then ${}^{gB}\nabla$ must have torsion and not all components of $T^i_{jk}$ vanish. By Lemma \ref{LemmaAllDimensionPDEs} $F$ satisfies the system (\ref{PDEsystem}). This system consists out of two equations, which are equivalent, as can be seen by contracting (\ref{PDEsystem}) with $y^k$. Hence, there is only one equation, let us choose the first one and rewrite it in the coordinates $(r,\theta)$:
\begin{align}\label{EquationLongCalculation}
0 = & F_{y^1 y^i} \Gamma^i -  F_{y^i} T^i_1 \nonumber\\
=& F_{y^1y^1} \Gamma^1 +F_{y^1y^2} \Gamma^2  - F_{y^1}T^1_{12}y^2 - F_{y^2} T^2_{12}y^2 \nonumber\\
=& (f+f_{\theta \theta}) \sin^2 \theta \big(\Gamma^1_{11}\cos^2 \theta + 2 \Gamma^1_{12} \cos \theta \sin \theta + \Gamma^1_{22} \sin^2 \theta\big) \nonumber\\
& -(f+f_{\theta \theta}) \cos\theta \sin \theta \big(\Gamma^2_{11}\cos^2 \theta + 2 \Gamma^2_{12} \cos \theta \sin \theta + \Gamma^2_{22} \sin^2 \theta\big)\nonumber\\
& +(\sin \theta f_\theta - \cos \theta f)T^1_{12} \sin \theta -(\cos \theta f_\theta + \sin \theta f)T^2_{12}\sin \theta\nonumber\\
=& \big[ (f+f_{\theta \theta}) \Big( - \Gamma^2_{11}\cos^3 \theta
+(\Gamma^1_{11}-2 \Gamma^2_{12})\cos^2 \theta  \sin \theta \\
&+ (2 \Gamma^1_{12}-\Gamma^2_{22}) \cos \theta \sin^2 \theta + \Gamma^1_{22} \sin^3 \theta\Big) \nonumber\\
& +(\sin \theta f_\theta - \cos \theta f)T^1_{12} 
-(\cos \theta f_\theta + \sin \theta f)T^2_{12} \big]\sin \theta.\nonumber
\end{align}
	Denote the factor of $(f+f_{\theta \theta})$ in the last formula by $P$, that is
	$$P:= K_3 \cos^3 \theta
+K_2\cos^2 \theta  \sin \theta +  K_1 \cos \theta \sin^2 \theta + K_0 \sin^3 \theta$$
	with coefficients \vspace{-1.3mm}
	 $$K_3 = - \Gamma^2_{11} \qquad
	K_2 = \Gamma^1_{11}-2 \Gamma^2_{12} \qquad 
	K_1 = 2 \Gamma^1_{12}-\Gamma^2_{22} \qquad
	K_0 = \Gamma^1_{22}.$$
	As the torsion is antisymmetric in the lower indices, for fixed $x \in M$ it is given by only one vector and we might assume without loss of generality that 
	$(T^1_{12},T^2_{12})=(1,0)$.
	Then the above equation (\ref{EquationLongCalculation}) implies
	\begin{equation}\label{EquationDimension2}
	(f+f_{\theta \theta})  P = - \sin \theta f_\theta + \cos \theta f.
	\end{equation}		 
	We now proof the theorem by showing that
	\begin{enumerate}[label=(\alph*)]
	\item if (\ref{EquationDimension2}) admits a fiber-global, strictly convex Finsler metric as a solution, then the coefficients $K_i$ must be of the form 
	\begin{equation}\label{EquationFormOfKs}
	K_3=\frac{1}{C} \quad
	K_2=-\frac{3A}{C} \quad
	K_1=\frac{3A^2}{C}+1 \quad
	K_0=-\frac{A(A^2+C)}{C}\vspace{-1mm}
	\end{equation}	
	for some constants $A, C \in \mathbb R$ with $C>0$. \label{PartShowCoefficientsForm}
	\item in this case, all solutions must be of Randers type. \label{PartMustBeRanders}
	\end{enumerate}
	\textbf{For \ref{PartShowCoefficientsForm}}, first recall that by Remark \ref{RemarkSymmetrization}, if equation (\ref{EquationDimension2}) admits a fiber-global Finsler metric $F(\theta)=r \cdot f(\theta)$ as a solution, then also its symmetrization $F_s(\theta)=\tfrac r 2 (f(\theta) + f(\theta + \pi))$ is a Finsler metric solution. Thus, we might assume that $F(\theta)=r \cdot f(\theta)$ is already symmetric.
	
	Define $g := \frac{-\sin \theta f_\theta + \cos \theta f}{P}$ and note that by equation (\ref{EquationDimension2}), it is $g= f+f_{\theta\theta}$ and thus $g$ must be defined everywhere and is positive by Lemma \ref{LemmaTwoDimensionalWhenStrictlyConvex}. We calculate that
$$g_\theta
= - \frac{P_\theta + \sin \theta }{P} g$$
and hence 
	\begin{equation}\label{EquationLNg}
	\big( \ln(g) \big)_\theta = \frac{g_\theta}{g}
	= - \frac{P_\theta + \sin \theta }{P}.
	\end{equation}
 	Defining the polynomial $p$ by
	$p(t):= K_3 t^3 + K_2 t^2 + K_1 t + K_0$ and using the cotangent function $\ctg$, we get for $\theta \in (0, \pi)$ that
	\begin{align*}
	P_\theta (\theta) &= \Big( p(\ctg \theta)\cdot \sin^3(\theta) \Big)_\theta\\
	&=p'(\ctg \theta)\ctg'\theta \sin^3 \theta + 3 p(\ctg \theta) \cos \theta \sin^2 \theta.
\end{align*}
	Now using $\ctg' \theta = - \frac{1}{\sin^2 \theta}$, equation (\ref{EquationLNg}) can be rewritten as
	\begin{equation}\label{EquationFormulaForDerivativeOfLNG}
	-\big( \ln g \big)_\theta(\theta)
	= \frac{\big(p'(\ctg \theta) -1 \big)\ctg' \theta}{p(\ctg \theta)} + 3 \ctg \theta.
	\end{equation}		
	Next, we exploit that the right hand side of this equation must be defined everywhere and its integral over $(0, \pi )$ must vanish by the $\pi$-periodicity of $g$.
	As the right hand side of (\ref{EquationFormulaForDerivativeOfLNG}) is not allowed to have a singularity and $\ctg (\theta)$ runs over all reals as $\theta \in (0, \pi)$, any root of $p$, that is $t \in \mathbb R$ with $p(t)=0$, must be a root of $p'-1$.
	Assume\footnote{The case $K_3=0$ is completely analogous, as then $p$ is a polynomial of degree two and $p'-1$ of degree one.} that $K_3 \not=0$, so that $p$ has exactly one real root $A$ and can be written as
	\begin{equation}\label{EquationFormOfPolynomial1}
	p(t)=K_3(t-A)\big( (t-B)^2 +C\big)
	\end{equation}		
	for some $A,B,C \in \mathbb R$ with $C >0$. Then $A$ must be also a root of $p'-1$, so for some $D \in \mathbb R$ we may write
	\begin{equation}\label{EquationFormOfPolynomial2}
	p'(t)-1=3K_3(t-A)(t-D).
	\end{equation}		
	Using the integral vanishing argument from above and that $\ctg$ is odd with respect to $\pi/2$, we have	
	$$0=\int_0^\pi -(\ln g)_\theta d\theta = \int_{-\infty}^\infty \frac{p'(t)-1}{p(t)} dt
	=  \int_{-\infty}^\infty\frac{3(t-D)}{(t-B)^2+C} dt,$$
	which implies that $B=D$. Expanding (\ref{EquationFormOfPolynomial1}) and (\ref{EquationFormOfPolynomial2}) and comparing coefficients, we get the relations
	$$K_2 = -K_3(2B+A), \quad K_1=K_3(B^2 + C+2AB) \quad K_0 = - K_3 A(B^2+C)$$
	$$2K_2 = -3K_3 (A+B) \qquad K_1 -1 = 3 K_3 AD.$$
	Combining the 1st and 4th, it follows that $A=B$. Combining the 2nd and 5th, it follows that $C=\frac{1}{K_3}$. Thus the coefficients $K_i$ are of the form (\ref{EquationFormOfKs}), as claimed.	
	\pagebreak
	
	\textbf{For \ref{PartMustBeRanders}}, let us notice that equation (\ref{EquationDimension2})is linear and can be solved for the 2nd order derivative $f_{\theta \theta}$ near all but finitely many points, and thus the space of $2\pi$-periodic solutions is a subset of a 2-dimensional vector space. The function $f(\theta)=\sin \theta$ is always a (non-Finsler) solution. Thus it is enough to find one more independent solution and any other must be a sum of the two.
	
	Let us determine for which coefficients $K_i$, the Riemannian norm \linebreak $f=\sqrt{g_{11} \cos^2 \theta + 2 g_{12} \cos \theta \sin \theta + g_{22} \sin^2 \theta}$ is a solution to the equation, where $(g_{ij})$ is a positive definite matrix. By direct computations, we obtain that
	$$- \sin \theta f_\theta + \cos \theta f = - F_{y^1}|_{r=1}
	=\frac{1}{f} \big( g_{11} \cos \theta + g_{12} \sin \theta \big)$$
	and $$f+f_{\theta \theta} = \big( \tfrac1{(y^2)^2}F_{y^1y^1}\big)|_{r=1}
	= \frac{g_{11}g_{22} - g_{12}^2}{f^3}.$$
	Plugging these into equation (\ref{EquationDimension2}) and multiplying with $\frac{f^3}{g_{11}g_{22}-g_{12}^2}$ gives
	$$P = \frac{(g_{11}\cos^2 \theta + 2g_{12} \cos \theta \sin \theta + g_{22} \sin^2 \theta)(g_{11} \cos \theta + g_{12} \sin \theta)}{g_{11}g_{22}-g_{12}^2}.$$
	Comparing the coefficients, we obtain that equation (\ref{EquationDimension2}) admits the Riemannian norm given by $g_{ij}$ as a solution, if and only if the coefficients $K_i$ are of the form
	\begin{equation}\label{EquationFormOfKsForRiemannian}{\small
	\begin{matrix}
	K_3= \frac{g_{11}^2}{g_{11}g_{22} - g_{12}^2} &
	K_2 = \frac{3g_{11}g_{12}}{g_{11}g_{22} - g_{12}^2}&
	K_1 = \frac{3g_{12}^2}{g_{11}g_{22} - g_{12}^2}+1 &
	K_0 = \frac{g_{12}g_{22}}{g_{11}g_{22} - g_{12}^2}
	\end{matrix}.}
	\end{equation}
	In this case, the most general solution of the equation is of the form
	$$c_1 \cdot \sqrt{g_{11} \cos^2 \theta + 2 g_{12} \cos \theta \sin \theta + g_{22} \sin^2 \theta} + c_2 \sin \theta, \qquad c_1,c_2 \in \mathbb R.$$
	In particular, all Finsler metrics, that are solutions, are Randers metrics.
	
	It remains to show that, when the equation admits a fiber-global solution and the coefficients $K_i$ are of the form (\ref{EquationFormOfKs}) as in (a), then the equation admits a Riemannian solution. 
	Indeed, for $K_i$ of the form (\ref{EquationFormOfKs}), the Riemannian metric given by the positive definite matrix
	$$g_{ij} = 
	\begin{pmatrix}
	1 & -A \\
	-A & A^2+C
	\end{pmatrix}
	$$
	is a solution to equation (\ref{EquationDimension2}), as can be seen by plugging this values in  (\ref{EquationFormOfKsForRiemannian}).

\end{proof}
	
	\bibliographystyle{plain}
	\bibliography{literature}

\end{document}